\documentclass[12pt]{article} 
\usepackage{latexsym, a4wide}
\usepackage{amsmath, rotating, color} 
\usepackage{amsfonts,amssymb, amsthm, mathtools} 
\usepackage{pictexwd} 
\usepackage{mathrsfs}
\usepackage[normalem]{ulem} 
\newcommand\unnumberedfootnote[1]{ %
  \let\temp=\thefootnote %
  \renewcommand{\thefootnote}{}%
  \footnote{#1}%
  \let\thefootnote=\temp%
  \addtocounter{footnote}{-1}}

\newcommand{\N}  {{\mathbb N}}
\newcommand{\R}  {{\mathbb R}}

\renewcommand{\L}  {{\mathcal L}}
\newtheorem{theorem}{Theorem}
\newtheorem{proposition}{Proposition}[section]
\newtheorem{lemma}[proposition]{Lemma}

\theoremstyle{definition}
\newtheorem{remark}[proposition]{Remark}
\newtheorem{example}[proposition]{Example}

\numberwithin{equation}{section}
\begin{document}
\title{Height and the total mass of the forest of genealogical trees of a large population with general competition.}
\author{ Le V.\and Pardoux  E.}

\maketitle
   
\section{Introduction}

Consider a continuous time branching process, which takes values either in $\N$ or in $\R_+$ (in the second case one speaks of a continuous
state branching process, and we shall consider only those such processes with continuous paths). Such processes can be used as models of population growth. However, in that context one might want to model  interactions between the individuals (e.g. competition for limited resources) so that we
 no longer have a branching process.  Such interactions can increase the number of births, or in contrary
increase the number of deaths. The popular logistic  competition has been considered in Le, Pardoux, Wakolbinger \cite{LPW}, while a much more general type of interaction appears in Ba, Pardoux \cite{BP2}. 

We will assume that for large population size the interaction is of the type of a competition, which limits the size of the population. 
One may then wonder in which cases the interaction is strong enough so that the extinction time (or equivalently the height of the forest of genealogical trees)
remains finite, as the number of ancestors tends to infinity, or even such that the length of the forest of genealogical trees (which in the case of continuous state  is rather called its total mass) remains finite, as the population size tends to infinity.

This question has been addressed in the case of a polynomial interaction in Ba, Pardoux \cite{BP1}. Here we want to generalize those results to a very general type of competition, and we will also show that whenever our condition enforces a finite extinction time (resp. total mass) for the process started with infinite mass, that random variable has some finite exponential moments.

Let us describe the two classes of models which we will consider.

We first describe the discrete state model.
Consider a population evolving in  continuous time with $m$ ancestors at time $t=0$, in which each individual, independently of the others,  gives birth to one child at a constant rate $\lambda $, and dies  after an exponential time with parameter $\mu $.  For each individual we superimpose additional birth and death rates due to interactions  with others  at a certain rate which depends upon the other individuals in the population. 
More precisely, given a function $f:\R_+\to\R$ which satisfies assumption (H1) below, whenever the total size of the population is $k$, the total additional birth rate due to interactions is 
$\sum_{j=1}^k(f(j)-f(j-1))^+$, while  the total additional death rate due to interactions is $\sum_{j=1}^k(f(j)-f(j-1))^-$. Let $X^m_t$ denote the population size at
time $t>0$, originating from $m$ ancestors at time 0. The above description is good enough for prescribing the evolution of $\{X^m_t,\ t\ge0\}$
with one value of $m$. There is a natural way to couple those evolutions for different values of $m$ which will be described in section 2 below, such that $m\mapsto  X^m_t$ is increasing for all $t\ge0$, a.s.

If we consider this population with $m=[Nx]$ ancestors at time $t=0$, replace $\lambda$ by $\lambda _N=2N$, $\mu$ by
 $\mu _N=2N$, $f$ by $f_N(x)=Nf(x/N)$, and define the weighted population
 size process $Z^N_t=N^{-1}X^N_t$,  it is shown in \cite{BP2} that $Z^N$ converges weakly to the unique solution of the SDE (see Dawson, Li \cite{DL})
 \begin{align}\label{DLeq}
 Z^x_t=x+\int_0^tf(Z^x_s)ds+2\int_0^t\int_0^{Z^x_s}W(ds,du),
\end{align}
where $W$ is space--time white noise on $\R_+\times \R_+$. This SDE couples the evolution of the various $\{Z^x_t,\ t\ge0\}$ jointly for all values of $x>0$.

We will use the fact that for a given value of $x>0$, there exists a standard Brownian motion $W$, such that
\begin{equation}\label{eq1}
Z^x_t=x+\int_0^t f(Z^x_s)ds+2\int_0^t\sqrt{Z^x_s}dW_s.
\end{equation}
There is a natural way of describing the genealogical tree of the discrete population. The notion of genealogical tree is discussed for the limiting continuous population as well in \cite{LPW, PW}, in terms of continuous random trees in the sense of Aldous \cite{AD}. Clearly  one can define the height $H^m$ and the length $L^m$ of the discrete forest of genealogical trees, as well as the height of the continuous ``forest of genealogical trees", equal to the lifetime $T^x$ of the process $Z^x$, and the total mass of the same forest of trees, given by $S^x := \int_0^{T^x} Z_t^x dt$.

Our assumption concerning the function $f$ will be

 \textbf{Hypothesis (H1):} $f\in C(\mathbb R_{+}, \mathbb R)$, $f(0)=0$, and there exists $\theta \ge 0$ such that 
 $$f(x+y)-f(y)\leq \theta x \qquad \forall x,y\geq 0.$$
 Note that the hypothesis (H1) implies that the function $\theta x- f(x)$ is increasing. In particular, we have
 $$f(x)\leq \theta x \qquad \forall x\geq 0.$$

The paper is organized as follows. Section $2$ studies the discrete case, i.e. the case of $\N$--valued processes, while section $3$ studies the continuous case, i.e. the case of
$\R$--valued processes. Each of those two sections starts with a subsection presenting
necessary preliminary material. The main results in the discrete case are Theorem $3$ and $4$, while the main results in the continuous case are Theorem $6, 7$ and $8$. Section 4 gives some examples to illustrate our results.

\begin{remark}
This remark aims at helping the reader to build his intuition about our results.
Take first a locally Lipschitz function $f : \R\to \R_+$, such that for simplicity $f(x) > 0$,
for all $x$, and consider the ODE $\dot x = f(x)$. It is easily seen that the solution $x$ explodes
in finite time iff
$\int_0^\infty  dx/f(x) < \infty$, and in that case, denoting $t_\infty$ the time of explosion, 
$\int_0^{t_\infty} x(t)dt < \infty$ iff $\int_0^\infty  xdx/f(x) < \infty$.

Reversing time, we deduce that if now $f(x) < 0$ for all $x$ (or all $x$ sufficiently large),
the same ODE has a solution which satisfies $x(t) \in \R$ for all $t \in (0; T]$ for some $T > 0$ and
$x(t)\to+\infty$ as $t \to 0$ (i.e. in a sense $x(0) = +\infty$) iff for some $M > 0$,
$\int_M^\infty  dx/|f(x)| < \infty$,
and that solution is locally integrable near t = 0 iff
$\int_M^\infty x dx/|f(x)| < \infty$.
The fact that these results can be extended to certain SDEs is essentially our argument
in the continuous population case, see section \ref{sec3} below. Once this is understood, it is clear
that similar results might be expected to hold true in the finite population case, which is
the content of section \ref{sec2}.
\end{remark}

\section{The discrete case}\label{sec2}
\subsection{Preliminaries}	
We consider a continuous time $\mathbb{Z}_{+}$--valued population process $\{X_t^m, t\geq 0, m \geq  1\}$, which starts at time zero from the initial condition $X_0^m=m$, i.e. $m$ is the number of ancestors of the whole population. 
$\{X_t^m, t\geq 0\}$ is a continuous time $\mathbb{Z}_{+}$--valued Markov process, which evolves as follows. If $X_t^m=0$, then $X_s^m=0$ for all $s\geq t$. While at state $k\geq 1$, the process
\begin{align*}
X_t^m \quad \text{jumps to} 
\begin{cases}
& k+1, \quad \text{at rate}\quad \lambda k+F^{+}(k)\\
& k-1, \quad \text{at rate}\quad \mu k+F^{-}(k),
\end{cases}
\end{align*}
where $f$ is a function satisfying (H1), $\lambda , \mu $ are positive constants, and
$$F^{+}(k):=\sum_{\ell=1}^{k} (f(\ell)- f(\ell-1))^{+}, \qquad F^{-}(k):=\sum_{\ell=1}^{k} (f(\ell)- f(\ell-1))^{-}.$$
 We now describe a joint evolution of all $\{X_t^m, t\geq 0\}_{m\geq 1}$, or in other words of the two-parameter process $\{X_t^m, t\geq 0, m \geq  1\}$, which is consistent with the above prescriptions. 
 Suppose that the $m$ ancestors are arranged  from left to right. The left/right order  is passed on to their offsprings: the daughters are placed on the right of their mothers  and  if at a time $t$ the individual $i$ is located at the right of individual $j$, then all the offsprings of $i$ after time $t$ will be placed on the right of all the offsprings of $j$. Since we have excluded multiple births at any given time, this means that the forest of genealogical trees of the population is a planar forest of trees, where the ancestor of the population $X^1_t$ is placed on the far left, the ancestor of $X^2_t-X^1_t$ immediately on his right, etc... Moreover, we draw the genealogical trees in such a way that distinct branches never cross. This defines in a non--ambiguous way an order from left to right within the population alive at each time $t$. 
 We decree that  each individual feels the interaction with the others placed on his left but  not with those on his right. Precisely, at any time $t$, the individual $i$ has an interaction death rate  equal to $\left( f(\L_i(t)+1)-f(\L_i(t))\right)^-$ or an interaction birth rate equal to $\left( f(\L_i(t)+1)-f(\L_i(t))\right)^+$, where $\L_i(t)$ denotes the number of individuals alive at time $t$ who are located on the left of $i$ in the above planar picture.
This means that the  individual $i$ is under attack by the others located at his left  if  $f(\L_i(t)+1)-f(\L_i(t)) <0$ while the interaction improves his fertility if $f(\L_i(t)+1)-f(\L_i(t))>0$.  
Of course, conditionally upon $\L_i(\cdot)$, the occurence of a ``competition death event''  or an ``interaction birth event"  for individual $i$ is independent  of the other birth/death events and of what happens to the other individuals. In order to simplify our formulas, we suppose moreover that  the first individual in the left/right order has a birth rate equal to $\lambda+f^+(1)$ and a death rate equal to  $\mu+f^-(1)$.

 \begin{remark} The functions $F^+$ and $F^-$ may look a bit strange.
 However, if $f$ is either increasing or decreasing, which is the case in particular if $f$ is linear, then $F^+=f^+$ and 
 $F^-=f^-$. 
 \end{remark}
 
 Define the height and length of the genealogical forest of trees by
\begin{equation*}
H^m= \inf \{t>0, X_t^m=0\}, \qquad L^m= \int_0^{H^m} X_t^m dt, \qquad \text{for} \quad m\geq 1.
\end{equation*}

Note that our coupling of the various $X^m$'s makes $H^m$ and $L^m$ a.s. increasing w.r. to $m$. We now study the limits of $H^m$ and $L^m$ as $m\rightarrow \infty $. We first recall some preliminary results on birth and death processes, which can be found in \cite{AWJ, Cat, KT}.\\

Let $Y$ be a birth and death process with birth rate $\lambda _n>0$ and death rate $\mu _n>0$ when in state $n, n\geq 1$. Let 
\begin{equation*}
A= \sum_{n\geq 1} \frac 1{\pi _n}, \qquad S= \sum_{n\geq 1} \frac 1{\pi _n} \sum_{k\geq n+1} \frac{\pi _k}{\lambda _k},
\end{equation*}
where 
\begin{equation*}
\pi _1=1, \qquad \pi _n = \frac{\lambda _1...\lambda _{n-1}\lambda _n}{\mu _2...\mu _{n}}, \quad n\geq 2.
\end{equation*}
We denote by $T_y^m$ the first time the process $Y$ hits $y\in [0,\infty )$ when starting from $Y_0=m$.
\begin{equation*}
T_y^m= \inf\{t>0: Y_t=y \mid Y_0=m\}.
\end{equation*}
We say that $\infty $ is an entrance boundary for $Y$ (see, for instance, Anderson \cite{AWJ}, section $8.1$) if there is $y>0$ and a time $t>0$ such that 
\begin{equation*}
\lim_{m\uparrow \infty }\mathbb P(T_y^m<t) >0.
\end{equation*}
We have the following result (see \cite{Cat}, Proposition $7.10$)
\begin{proposition}\label{prop2.1}
The following are equivalent:
\begin{itemize}
\item[\rm{1)}] $\infty $ is an entrance boundary for $Y$.
\item[\rm{2)}] $A=\infty , S<\infty $.
\item[\rm{3)}] $\lim_{m\uparrow \infty }\mathbb E(T_0^m)<\infty $.
\end{itemize}
\end{proposition}
We now want to apply the above result to the process $X_t^m$, in which case $\lambda _n= \lambda n+F^{+}(n), \mu _n= \mu n+F^{-}(n), n\geq 1$. We will need the following lemmas.
\begin{lemma}\label{le2.2}
Let $f$ be a function satisfying (H1), $a\in \mathbb R$ be a constant. If there exists $a_0>0$ such that $f(x)\not= 0, f(x)+ ax \not= 0$ for all $x\geq a_0$, then we have that 
\begin{equation*}
\int_{a_0}^\infty \frac 1{\mid f(x)\mid } dx <\infty \Leftrightarrow \int_{a_0}^\infty \frac 1{\mid ax+ f(x)\mid } dx <\infty ,
\end{equation*}
and when those equivalent conditions are satisfied, we have
$$\lim_{x\rightarrow \infty } \frac{f(x)}{x}= -\infty .$$
\end{lemma}
\begin{proof}
We need only show that
$$\int_{a_0}^\infty \frac 1{\mid f(x)\mid } dx <\infty \Rightarrow \int_{a_0}^\infty \frac 1{\mid ax+ f(x)\mid } dx <\infty.$$
Indeed, this will imply the same implication for pair $f'(x)=f(x)+ax$, $f'(x)-ax$, which is the conversed result.
Because $f(x)\leq \theta x$ for all $x\geq 0$, we can easily deduce from $\int_{a_0}^\infty \frac 1{\mid f(x)\mid } dx <\infty$ that $$f(x)<0 \qquad \forall x\geq a_0.$$
Let $\beta $ be a constant such that $\beta >\theta $. We have
$$\int_{a_0}^\infty \frac 1{ \beta  x- f(x)} dx<\int_{a_0}^\infty \frac 1{ -f(x)} dx<\infty .$$
It implies that $$\lim_{x\rightarrow \infty } \int_x^{2x} \frac 1{\beta  u-f(u)} du =0.$$
But since the function $x\mapsto  \beta x-f(x)$ is increasing,
$$\int_x^{2x} \frac 1{\beta  u-f(u)} du\geq \big(2\beta -\frac{f(2x)}{x}\big)^{-1}.$$ 
We deduce that $\lim_{x\rightarrow \infty } \frac{f(x)}{x}=-\infty $. Hence there exists $a_1>a_0$ such that $f(x)<-2|a|x$ for all $x\geq a_1$. The result follows from 
$$\int_{a_1}^\infty \frac 1{\mid ax+f(x)\mid } dx< \int_{a_1}^\infty \frac 2{ -f(x)} dx <\infty .$$
\end{proof}
\begin{lemma}\label{le2.3}
Let $f$ be a function satisfying (H1). For all $n\geq 1$ we have the two inequalities
\begin{align*}
 F^{+}(n)&\leq \theta n\\
 -f(n)\leq F^{-}(n)&\leq \theta n-f(n).
\end{align*}
\end{lemma}
\begin{proof}
The result follows from the facts that for all $n\geq 1$
\begin{align*}
(f(n)-f(n-1))^{+}&\leq \theta \\
 (f(n)-f(n-1))^{-}&\geq f(n-1)-f(n)\\
F^{-}(n)-F^{+}(n)&=-f(n).
\end{align*}
\end{proof}
\begin{proposition}\label{prop2.4}
Assume $f$ is a function satisfying (H1) and there exists $a_0>0$ such that $f(x)\not= 0$ for all $x\geq a_0$. Then $\infty $ is an entrance boundary for $X$ if and only if
\begin{equation*}
\int_{a_0}^\infty \frac 1{\mid f(x)\mid } dx <\infty.
\end{equation*}
\end{proposition}
\begin{proof}
If $\int_{a_0}^\infty \frac 1{\mid f(x)\mid } dx =\infty$, then (recall that since $\mu >0, (\mu+\theta)x-f(x)$ is non--negative and increasing) $$\int_{a_0}^\infty \frac 1{(\mu +\theta )x- f(x)} dx =\infty ,$$ by Lemma \ref{le2.2}. In this case,
\begin{align*}
S&\geq \sum_{n\geq 1} \frac{\pi _{n+1}}{\lambda _{n+1}\pi _n}\\
&= \sum_{n\geq 1}\frac 1{\mu _{n+1}}\\
&= \sum_{n\geq 2} \frac 1{\mu n+ F^{-}(n)}\\
&\geq \sum_{n\geq 2} \frac 1{(\mu +\theta )n- f(n)}\\
&= \infty .
\end{align*}
Therefore, $\infty $ is not an entrance boundary for $X$, by Proposition \ref{prop2.1}. On the other hand, if $\int_{a_0}^\infty \frac 1{\mid f(x)\mid } dx <\infty $, then $\lim_{x\rightarrow \infty } \frac{f(x)}{x}=-\infty$, by Lemma \ref{le2.2}. By Lemma \ref{le2.3} we have
\begin{equation*}
\lim_{n\rightarrow \infty }\frac{\pi _{n+1}}{\pi _n}= \lim_{n\rightarrow \infty } \frac{\lambda n+F^{+}(n)}{\mu n+F^{-}(n)} \leq \lim_{n\rightarrow \infty } \frac{(\lambda +\theta )n}{\mu n-f(n)}= 0,
\end{equation*}
so that 
\begin{equation*}
A= \sum_{n\geq 1} \frac 1{\pi _n}= \infty .
\end{equation*}
Set $a_n= \lambda _n/ \mu_n $, then there exists $ n_0\geq 1$ such that $ a_n<1$ for all $n\geq n_0$. The inequality of arithmetic and geometric means states that for all $m>0$ and $x_1, x_2, ..., x_m>0$,
$$\frac{x_1+x_2+ ...+ x_m}{m}\geq \sqrt[m]{x_1 x_2 ... x_m},$$
so that for all $k>n>0$,
$$a_{n+1}^{k-n}+...+ a_k^{k-n}\geq (k-n) a_{n+1}...a_k. $$
Then
\begin{align*}
\sum_{n\geq n_0} \frac 1{\pi _n} \sum_{k\geq n+1} \frac{\pi _k}{\lambda _k}
&\leq \frac 1{\lambda }\sum_{n\geq n_0}  \sum_{k\geq n+1} \frac 1{k} a_{n+1}...a_k \\
&\leq \frac 1{\lambda }\sum_{n\geq n_0}  \sum_{k\geq n+1} \frac 1{k(k-n)} (a_{n+1}^{k-n}+...+ a_k^{k-n})\\
&= \frac 1{\lambda }\sum_{k\geq n_0+1}  \sum_{n=1}^{k-n_0} \frac 1{k n} (a_{k-n+1}^{n}+...+ a_k^{n})\\
&= \frac 1{\lambda } \sum_{i\geq n_0+1}\sum_{n\geq 1} a_i^n \sum_{k=i}^{n-1+i} \frac 1{k n}\\
&\leq \frac 1{\lambda }\sum_{i\geq n_0+1}\sum_{n\geq 1} \frac{a_i^n}{i} \\
&= \frac 1{\lambda }\sum_{i\geq n_0+1} \frac{a_i}{i(1-a_i)} \\
&= \frac 1{\lambda }\sum_{i\geq n_0+1} \frac{\lambda _i}{i(\mu _i-\lambda _i)} \\
&= \sum_{i\geq n_0+1} \frac{\lambda i+F^{+}(i)}{\lambda i(\mu i -\lambda i+F^{-}(i)-F^{+}(i))} \\
&\leq \frac{\lambda +\theta }{\lambda } \sum_{i\geq n_0+1} \frac 1{\mu i-\lambda i- f(i)}\\
&<\infty,
\end{align*}
where we have used Lemma \ref{le2.2} to conclude.
Hence $S<\infty $. The result follows from Proposition \ref{prop2.1}.
\end{proof}
We can now prove
\begin{theorem}\label{th1}
Assume $f$ is a function satisfying $(H1)$ and there exists $a_0>0$ such that $f(x)\not= 0$ for all $x\geq a_0$. We have
\begin{itemize}
\item[\rm{1)}] If $\int_{a_0}^\infty \frac 1{\mid f(x)\mid } dx =\infty$, then
 $$\sup_{m>0} T_0^m = \infty \quad a.s.$$
\item[\rm{2)}] If  $\int_{a_0}^\infty \frac 1{\mid f(x)\mid } dx <\infty$, then
 \begin{equation*}
\mathbb{E}\big(\sup_{m>0} T_0^m \big) < \infty.
\end{equation*}
\end{itemize}
\end{theorem}
\begin{proof}
If $\int_{a_0}^\infty \frac 1{\mid f(x)\mid } dx =\infty$, then by Proposition \ref{prop2.4}, $\infty $ is not an entrance boundary for $X$. It means that for all $t>0$,
\begin{equation*}
\lim_{m\uparrow \infty } \mathbb P (T_0^m <t) = 0.
\end{equation*}
Hence for all $t>0$, since $m\to T^m_0$ is increasing a.s.,
\begin{equation*}
\mathbb P ( \sup_{m>0} T_0^m <t) = 0,
\end{equation*}
hence
\begin{equation*}
\sup_{m>0} T_0^m =\infty  \qquad a.s.
\end{equation*}
The second part of the theorem is a consequence of Proposition \ref{prop2.4} and Proposition \ref{prop2.1}.
\end{proof}
\begin{remark}\label{re2.5}
The first part of Theorem \ref{th1} is still true when $\lambda_n =0, n\geq 1$. In fact, in this case we have
$$T_0^m \doteq \sum_{n=1}^m \theta _n,$$
where $\doteq $ denotes equality in law, $\theta _n$ represents the first passage time from state $n$ to state $n-1$,
$$\theta _n = \inf \{ t>0: X_t= n-1 \mid X_0=n\}.$$
Recalling the fact that $\theta _n$ is exponentially distributed with parameter $\mu n+ F^{-}(n)$, we have (see Lemma 4.3, Chapter 7 in \cite{PE} )
$$\sup_{m>0} T_0^m = \infty \quad a.s.\quad \Leftrightarrow \quad \sum_{n=1}^\infty \frac 1{\mu n+F^{-}(n)} =\infty .$$
The result follows by Lemma \ref{le2.2} and Lemma \ref{le2.3}.
\end{remark}
Here a question arises: in the case $\int_{a_0}^\infty \frac 1{\mid f(x)\mid } dx <\infty$, whether higher moments of $\sup_{m>0} T_0^m$ are also finite or not. We will see that the answer is Yes. Indeed, we can prove that it has some finite exponential moments.
\begin{theorem}\label{th2}
Suppose that $f$ is a function satisfying (H1) and there exists $a_0>0$ such that $f(x)\not= 0$ for all $x\geq a_0$. If $\int_{a_0}^\infty \frac 1{\mid f(x)\mid } dx <\infty$ we have
\begin{itemize}
\item[\rm{1)}] For any $a>0$, there exists $y_a\in \mathbb Z_+$ such that 
 $$\sup_{m>y_a} \mathbb E \big(e^{aT_{y_a}^m }\big) < \infty .$$
\item[\rm{2)}] There exists some positive constant $c$ such that 
 \begin{equation*}
\sup_{m>0} \mathbb E \big(e^{cT_{0}^m }\big) < \infty .
\end{equation*}
\end{itemize}
\end{theorem}
\begin{proof}
\begin{itemize}
\item[\rm{1)}]
There exists $n_a\in \mathbb Z_+$ large enough so that 
$$\sum_{n=n_a}^\infty \frac 1{\pi _n} \sum_{k\geq n+1} \frac {\pi _k}{\lambda _k} \leq \frac 1{a}.$$
Let $J$ be the nonnegative increasing function defined by
$$J(m) := \sum_{n=n_a}^{m-1} \frac 1{\pi _n} \sum_{k\geq n+1} \frac {\pi _k}{\lambda _k}, \qquad m\geq n_a+1.$$
Set now $y_a=n_a+1$. Note that $\sup_{m>y_a} T_{y_a}^m <\infty $ a.s., then for any $m>y_a$ we have
$$J(X_{t\wedge T_{y_a}^m}^m) -J(m) - \int_0^{t\wedge T_{y_a}^m} AJ(X_s^m) ds$$
is a martingale, where $A$ is the generator of the process $X_t^m$ which is given by
$$Ag(n) = \lambda _n (g(n+1)-g(n)) + \mu _n (g(n-1)-g(n)), \qquad n\geq 1,$$
for any $\mathbb R_+$-valued, bounded function $g$. Therefore, by Ito's formula
$$e^{a (t\wedge T_{y_a}^m)} J(X_{t\wedge T_{y_a}^m}^m) -J(m)- \int_0^{t\wedge T_{y_a}^m} e^{a s}(aJ(X_s^m)+ AJ(X_s^m)) ds$$
is also a martingale. It implies that 
$$ \mathbb E \Big( e^{a (t\wedge T_{y_a}^m)} J(X_{t\wedge T_{y_a}^m}^m)\Big)= J(m)+ \mathbb E \Big( \int_0^{t\wedge T_{y_a}^m} e^{a s}(aJ(X_s^m)+ AJ(X_s^m)) ds \Big).$$
We have for $m>y_a$, $J(X_s^m)<J(\infty )\leq \frac 1{a} \quad \forall s\leq T_{y_a}^m$, and for any $n\geq y_a$,
\begin{align*}
AJ(n)&= \lambda _n (J(n+1)-J(n)) + \mu _n (J(n-1)-J(n))\\
&= \lambda _n \frac 1{\pi _n} \sum_{k\geq n+1} \frac {\pi _k}{\lambda _k} - \mu _n \frac 1{\pi _{n-1}} \sum_{k\geq n} \frac {\pi _k}{\lambda _k} \\
&= \frac {\mu _2 ...\mu _n}{\lambda _1 ...\lambda _{n-1}} \sum_{k\geq n+1} \frac {\pi _k}{\lambda _k}- \frac {\mu _2 ...\mu _n}{\lambda _1 ...\lambda _{n-1}} \sum_{k\geq n} \frac {\pi _k}{\lambda _k} \\
&= -\frac {\mu _2 ...\mu _n}{\lambda _1 ...\lambda _{n-1}} \frac {\pi _n}{\lambda _n} \\
&= -1.
\end{align*}
So that
$$\mathbb E \Big( e^{a (t\wedge T_{y_a}^m)} J(X_{t\wedge T_{y_a}^m}^m)\Big)\leq J(m) \qquad \forall m>y_a.$$
But $J$ is increasing, hence for any $m>y_a$ one gets
$$0<J(y_a)\leq J(m)<J(\infty )\leq \frac 1{a}.$$
From this we deduce that
$$\mathbb E \Big( e^{a (t\wedge T_{y_a}^m)} \Big) \leq \frac 1{a J(y_a)}\qquad \forall m>y_a.$$
Hence
$$\mathbb E \Big( e^{a T_{y_a}^m} \Big) \leq \frac 1{a J(y_a)}\qquad \forall m>y_a,$$
by the monotone convergence theorem. The result follows.
\item[\rm{2)}] Using the first result of the theorem, there exists a constant $M\in \mathbb Z_+$ such that 
\begin{equation*}
\sup_{m>M} \mathbb{E} \big( e^{ T_{M}^m}\big) <\infty ,
\end{equation*}
or $\mathbb{E} \big( e^{ T_{M}}\big) <\infty$, where $T_M := \sup_{m>M} T_M^m$.

Given any fixed $T>0$, let $p$ denote the probability that starting from $M$ at time $t=0, X$ hits zero before time $T$. Clearly $p>0$. Let $\zeta $ be a geometric random variable with success probability $p$, which is defined as follows. Let $X$ start from $M$ at time $0$. If $X$ hits zero before time $T$, then $\zeta =1$. If not, we look the position $X_T$ of $X$ at time $T$.\\
If $X_T>M$, we wait until $X$ goes back to $M$. The time needed is stochastically dominated by the random variable $T_M$, which is the time needed for $X$ to descend to $M$, when starting from $\infty $. If however $X_T\leq M$, we start afresh from there, since the probability to reach zero in less than $T$ is greater than or equal to $p$, for all starting points in the interval $(0, M]$. \\
So either at time $T$, or at time less than $T+T_M $, we start again from a level which is less than or equal to $M$. If zero is reached during the next time interval of length $T$, then $\zeta =2$... Repeating this procedure, we see that $\sup_{m>0}T_0^m$ is stochastically dominated by
\begin{equation*}
\zeta  T+ \sum_{i=1}^{\zeta } \eta _i,
\end{equation*}
where the random variables $\eta _i$ are i.i.d, with the same law as $T_M $, globally independent of $\zeta $. We have
\begin{align*}
\sup_{m>0} \mathbb{E} \big( e^{c T_{0}^m}\big) &\leq \mathbb{E} \big( e^{c (\zeta  T+ \sum_{i=1}^{\zeta  } \eta _i)}\big) \\
&\leq \sqrt{\mathbb{E} \big( e^{2c \zeta  T}\big)} \sqrt{\mathbb{E} \big( e^{2c \sum_{i=1}^{\zeta  } \eta _i }\big)}.
\end{align*}
Since $\zeta  $ is a geometric$(p)$ random variable, then 
\begin{equation*}
\mathbb{E} \big( e^{2c \zeta  T}\big) = \frac p{1-p} \sum_{k=1}^\infty \big( e^{2cT} (1-p) \big)^k <\infty ,
\end{equation*}
provided that $c< -\log(1-p)/ 2T$.\\
Moreover, we have
\begin{align*}
\mathbb{E} \big( e^{2c \sum_{i=1}^{\zeta  } \eta _i }\big) &=\sum_{k=1}^\infty \mathbb{E} \big( e^{2c \sum_{i=1}^{k} \eta _i }\big) \mathbb P (\zeta  =k) \\
&= \sum_{k=1}^\infty \big[ \mathbb{E} \big( e^{2c T_M }\big) \big]^{k} \mathbb P (\zeta  =k) \\
&= \frac p{1-p} \sum_{k=1}^\infty \big[ \mathbb{E} \big( e^{2c T_M }\big)(1-p) \big]^{k}.
\end{align*}
Since $\mathbb{E} \big( e^{ T_{M}}\big) <\infty$, it follows from the monotone convergence theorem that $\mathbb{E} \big( e^{2c T_M }\big) \rightarrow 1$ as $c \rightarrow 0$. Hence we can choose $0<c< -\log (1-p)/2T$ such that 
\begin{equation*}
\mathbb{E} \big( e^{2c T_M }\big)(1-p) <1,
\end{equation*}
in which case $\mathbb{E} \big( e^{2c \sum_{i=1}^{\zeta  } \eta _i }\big)<\infty $.\\
Then $\sup_{m>0} \mathbb{E} \big( e^{c T_{0}^m}\big)<\infty $. The result follows.
\end{itemize}
\end{proof}
\subsection{Height and length of the genealogical forest of trees in the discrete case}
The following result follows from Theorem \ref{th1} and Theorem \ref{th2}
\begin{theorem}\label{th3}
Suppose that $f$ is a function satisfying (H1) and there exists $a_0>0$ such that $f(x)\not= 0$ for all $x\geq a_0$. We have
\begin{itemize}
\item[\rm{1)}] If $\int_{a_0}^\infty \frac {1}{\mid f(x)\mid } dx =\infty$, then
 $$\sup_{m>0} H^m = \infty \quad a.s.$$
\item[\rm{2)}] If $\int_{a_0}^\infty \frac {1}{\mid f(x)\mid } dx <\infty$, then
$$\sup_{m>0} H^m < \infty \quad a.s.,$$
and moreover, there exists some positive constant $c$ such that 
 \begin{equation*}
\sup_{m>0} \mathbb E \big(e^{c H^m }\big) < \infty .
\end{equation*}
\end{itemize}
\end{theorem}
Concerning the length of the genealogical tree we have
\begin{theorem}\label{th4}
Suppose that the function $\frac{f(x)}{x}$ satisfies (H1) and there exists $a_0>0$ such that $f(x)\not= 0$ for all $x\geq a_0$. We have
\begin{itemize}
\item[\rm{1)}] If $\int_{a_0}^\infty \frac {x}{\mid f(x)\mid } dx =\infty$, then
 $$\sup_{m>0} L^m = \infty \quad a.s.$$
\item[\rm{2)}] If $\int_{a_0}^\infty \frac {x}{\mid f(x)\mid } dx <\infty$, then
$$\sup_{m>0} L^m < \infty \quad a.s.,$$
and moreover, there exists some positive constant $c$ such that 
 \begin{equation*}
\sup_{m>0} \mathbb E \big(e^{c L^m }\big) < \infty .
\end{equation*}
\end{itemize}
\end{theorem}
To prove Theorem \ref{th4} we need the following result, which is Theorem 1 in Bhaskaran \cite{BBG}.
\begin{proposition}\label{prop2.7}
Let $Y^i$ be a birth and death process with birth rates $\{\lambda _n^{(i)}\}_{n\geq 1}$ and death rates $\{\mu _n^{(i)}\}_{n\geq 1}$($i = 1,2$), where $\lambda _n^{(i)}$ and $\mu _n^{(i)}$ satisfy the condition
\begin{equation}\label{eq2.1}
\sum_{n\geq 1} \frac 1{\pi _n} \sum_{k=1}^n \frac{\pi _k}{\lambda _k}=\infty .
\end{equation}
Suppose that
\begin{equation*}
\lambda _n^{(1)}\geq \lambda _n^{(2)}\quad \text{and}\quad  \mu _n^{(1)} \leq \mu _n^{(2)},\qquad  n\geq 1.
\end{equation*}
Then one can construct two processes $\tilde{Y^1}$ and $\tilde{Y^2}$ on the same probability space such that $\{\tilde{Y^i}(k), k\geq 0\}$ and $\{Y^i(k), k\geq 0\}$ have the same law for $i=1,2$, and $\tilde{Y^1}(k)\geq \tilde{Y^2}(k)$ a.s. for all $k\geq 0$.
\end{proposition}
\begin{remark}\label{re2.8}
\begin{itemize}
\item[\rm{1)}] Condition \eqref{eq2.1} implies that the birth and death process does not explode in finite time a.s.Note that 
\begin{align*}
\sum_{n\geq 1} \frac 1{\pi _n} \sum_{k=1}^n \frac{\pi _k}{\lambda _k}&\geq \sum_{n\geq 1} \frac 1{\pi _n}\times \frac{\pi _n}{\lambda _n}\\
&= \sum_{n\geq 1} \frac 1{\lambda  _n}.
\end{align*}
Then \eqref{eq2.1} is satisfied if there exists a constant $\gamma  > 0$ such that
$$\lambda _n\leq \gamma n,\qquad \forall n\geq 1.$$
\item[\rm{2)}] Proposition \ref{prop2.7} is still true when $\lambda _n^2=0, n\geq 1$. In fact, the proof of Bhaskaran (as given in \cite{BBG}) still works in this case.
\end{itemize}
\end{remark}
Now we will apply Proposition \ref{prop2.7} to prove Theorem \ref{th4}. In the proof, we will not bother to check condition \eqref{eq2.1}, which is obviously satisfied here.\\
\\
{\bf Proof of Theorem \ref{th4}} 
\begin{itemize}
\item[\rm{1)}] Let
$$f_1(n):=\frac{f(n)}{n},\qquad F_1^{-}(n):= \sum_{k=1}^{n} (f_1(k)-f_1(k-1))^{-}, \qquad n\geq 1.$$
By Lemma \ref{le2.6} below we have for all $n\geq 1,$
\begin{align*}
\mu _n= \mu n+ F^{-}(n)&\leq \mu n+2\theta n^2 - f(n)\\
&\leq (\mu +2\theta )n^2 -\frac{f(n)}{n} n\\
&\leq (\mu +2\theta )n^2 + F_1^{-}(n)n.
\end{align*}
Let $X^{1,m}$ be a birth and death process which starts from $X^{1,m}_0=m$, with birth rate $\lambda^1 _n=0$ and death rate 
$\mu^1 _n=(\mu +2\theta )n^2 + F_1^{-}(n)n$ when in state $n, n\geq 1$. From Proposition \ref{prop2.7} we deduce that for all $m\geq 1,$
\begin{equation*}
X^m\geq X^{1,m} (\text{in dist.}), \quad H^m\geq H^{1,m}(\text{in dist.}), \quad L^m\geq L^{1,m} (\text{in dist.}),
\end{equation*}
and moreover, since both $m\to L^m$ and $m\to L^{1,m}$ are a.s. increasing,
$$\sup_{m>0} L^m\geq \sup_{m>0} L^{1,m} (\text{in dist.}),$$
where $H^{1,m}, L^{1,m}$ are the height and the length of the genealogical tree of the population $X^{1,m}$, respectively.

We now use a random time-change to transform the length of a forest of genealogical trees into the height of another forest of genealogical trees, so that we can apply Theorem \ref{th1}. We define
\begin{equation*}
A_t^{1,m}:= \int_0^t X_r^{1,m} dr, \qquad \eta _t^{1,m} = \inf\{s>0, A_s^{1,m}>t\},
\end{equation*}
and consider the process $U^{1,m}:= X^{1,m}\circ \eta ^{1,m}$. Let $S^{1,m}$ be the stopping time defined by
\begin{equation*}
S^{1,m}= \inf\{r>0, U_r^{1,m}=0\},
\end{equation*}
then we have 
$$S^{1,m}= \int_0^{H^{1,m}} X_r^{1,m} dr=L^{1,m} \qquad a.s.$$
The process $X^{1,m}$ can be expressed using a standard Poisson processes $P$, as
\begin{equation*}
X_t^{1,m} = m- P\Bigg(\int_0^t [(\mu+2\theta ) (X_r^{1,m})^2+ F_1^{-}(X_r^{1,m})X_r^{1,m}] dr\Bigg).
\end{equation*}
Consequently the process $U^{1,m}$ satisfies
\begin{equation*}
U_t^{1,m} = m- P\Bigg(\int_0^t [(\mu+2\theta ) U_r^{1,m}+ F_1^{-}(U_r^{1,m})] dr\Bigg).
\end{equation*}
Applying Theorem \ref{th1} and Remark \ref{re2.5} we have
$$\sup_{m>0} L^{1,m}= \sup_{m>0} S^{1,m}=\infty \qquad a.s.,$$
hence $\sup_{m>0} L^m=\infty $ a.s. The result follows.
\item[\rm{2)}] For the second part of the theorem, we note that in the case $\int_{a_0}^\infty \frac {x}{\mid f(x)\mid } dx <\infty$, we have $\frac{f(x)}{x^2}\rightarrow -\infty $ as $x\rightarrow \infty $, by Lemma \ref{le2.2}. Then there exists a constant $u>0$ such that for all $n\geq u$ (using again Lemma \ref{le2.6}),
$$\mu n +F^{-}(n)\geq -f(n)\geq \theta n^2- \frac{f(n)}{2}.$$
We can choose $\varepsilon \in (0,1)$ such that for all $1\leq n\leq u$
$$\mu n \geq \varepsilon (\theta n^2- \frac{f(n)}{2}).$$
It implies that for all $n\geq 1$,
$$\mu n +F^{-}(n)\geq \varepsilon (\theta n^2- \frac{f(n)}{2}).$$
Let $X^{2,m}$ be a birth and death process which starts from $X^{2,m}_0=m$, with birth rate $\lambda^2 _n=(\lambda +2\theta )n^2$ and death rate $\mu^2 _n=\varepsilon (\theta n^2- \frac{f(n)}{2})$ when in state $n, n\geq 1$. From Lemma \ref{le2.6} and Proposition \ref{prop2.7} we deduce that for all $m\geq 1,$
\begin{equation*}
X^m\leq  X^{2,m} (\text{in dist.}), \quad H^m\leq  H^{2,m}(\text{in dist.}), \quad L^m\leq  L^{2,m} (\text{in dist.}),
\end{equation*}
where $H^{2,m}, L^{2,m}$ are the height and the length of the genealogical tree of the population $X^{2,m}$, respectively. We define
\begin{equation*}
A_t^{2,m}:= \int_0^t X_r^{2,m} dr, \qquad \eta _t^{2,m} = \inf\{s>0, A_s^{2,m}>t\},
\end{equation*}
and consider the process $U^{2,m}:= X^{2,m}\circ \eta ^{2,m}$. Let $S^{2,m}$ be the stopping time defined by
\begin{equation*}
S^{2,m}= \inf\{r>0, U_r^{2,m}=0\},
\end{equation*}
then we have 
$$S^{2,m}= \int_0^{H^{2,m}} X_r^{2,m} dr=L^{2,m} \qquad a.s.$$
Denote $f_2(x):=\frac{\varepsilon }2 (\frac{f(x)}{x}-\theta x)$, then $f_2$ is a negative and decreasing function, so that for all $n\geq 1,$
$$F_2^{+}(n):= \sum_{k=1}^{n} (f_2(k)-f_2(k-1))^{+}=0, \qquad F_2^{-}(n):= \sum_{k=1}^{n} (f_2(k)-f_2(k-1))^{-}=-f_2(n).$$ 
The process $X^{2,m}$ can be expressed using two mutually independent standard Poisson processes $P_1$ and $P_2$, as
\begin{equation*}
X_t^{2,m} = m+ P_1\Bigg(\int_0^t [(\lambda +2\theta )(X_r^{2,m})^2]dr \Bigg) - P_2\Bigg(\int_0^t [\frac{\varepsilon \theta }2 (X_r^{2,m})^2+ F_2^{-}(X_r^{2,m})X_r^{2,m}] dr\Bigg).
\end{equation*}
Consequently the process $U^{2,m}$ satisfies
\begin{equation*}
U_t^{2,m} = m+ P_1\Bigg(\int_0^t [(\lambda +2\theta )U_r^{2,m} +F_2^{+}(U_r^{2,m})]dr \Bigg) - P_2\Bigg(\int_0^t [\frac{\varepsilon \theta }2 U_r^{2,m}+ F_2^{-}(U_r^{2,m})] dr\Bigg).
\end{equation*}
By Theorem \ref{th2}, there exists some positive constant $c$ such that
\begin{equation*}
\sup_{m>0} \mathbb E \big(e^{c L^{2,m} }\big) = \sup_{m>0} \mathbb E \big(e^{c S^{2,m} }\big)< \infty ,
\end{equation*}
hence $$\sup_{m>0} \mathbb E \big(e^{c L^{m} }\big)\leq \sup_{m>0} \mathbb E \big(e^{c L^{2,m} }\big)<\infty .$$ The result follows.
$$\qquad \qquad \qquad \qquad \qquad \qquad \qquad \qquad \qquad \qquad \qquad \qquad \qquad \qquad \qquad \qquad \qquad \qquad \square$$
\end{itemize}
$\quad$ It remains to prove
\begin{lemma}\label{le2.6}
Suppose that the function $\frac{f(x)}{x}$ satisfies (H1). For all $n\geq 1$ we have the following inequalities
\begin{align*}
F^{+}(n)&\leq 2\theta n^2,\\
-f(n)\leq F^{-}(n)&\leq 2\theta n^2-f(n).
\end{align*}
\end{lemma}
\begin{proof}
Note that for all $k\geq 1$,
\begin{align*}
(f(k)-f(k-1))^{+}&=\Big( (k-1)(\frac{f(k)}{k}-\frac{f(k-1)}{k-1})+\frac{f(k)}{k}\Big)^{+}\\
&\leq (k-1)\Big(\frac{f(k)}{k}-\frac{f(k-1)}{k-1}\Big)^{+} + \Big(\frac{f(k)}{k}\Big)^{+}\\
&\leq 2\theta k.
\end{align*}
Then $$F^{+}(n)\leq \sum_{k=1}^n 2\theta k = \theta n(n+1)\leq 2\theta n^2.$$
The second result now follows from the fact that for all $n\geq 1$
\begin{align*}
 (f(n)-f(n-1))^{-}&\geq f(n-1)-f(n)\\
F^{-}(n)-F^{+}(n)&=-f(n).
\end{align*}
\end{proof}
\section{The continuous case}\label{sec3}
\subsection{Preliminaries}	
We now consider the $\mathbb{R}_{+}$--valued two--parameter stochastic process $\{Z_t^x, t\geq 0, x \geq  0\}$ which solves the SDE \eqref{DLeq}, where the function $f$ satisfies (H1). We note that this coupling of the  $\{Z^x_t,\ t\ge0\}$'s for various $x$'s is consistent with that used in the discrete population case in the sense that as $N\to\infty$,
$$
\{N^{-1}X^{\lfloor Nx\rfloor}_t,\ t\ge0, x>0\}\Rightarrow\{Z^x_t,\ t\ge0, x>0\},$$
see \cite{BP2}, where the topology for which this is valid is made precise.

According again to \cite{BP2}, the process $\{Z_{.}^x, x\geq 0\}$ is a Markov process with values in $C(\mathbb R_+ , \mathbb R_+)$, the space of continuous functions from $\mathbb R_+$ into $\mathbb R_+$, starting from $0$ at $x = 0$. Moreover, we have that whenever $0<x\leq y, Z_t^y\geq Z_t^x$ for all $t\geq 0$ a.s. For $x>0$, define $T^x$ the extinction time of the process $Z^x$ (it is also called the height of the process $Z^x$) by
\begin{align*}
T^x &= \inf\{t>0, Z_t^x=0\}.
\end{align*}
And define $S^x$ the total mass of $Z^x$ by
\begin{align*}
S^x &= \int_0^{T^x} Z_t^x dt.
\end{align*}
We next study the limits of $T^x$ and $S^x$ as $x\rightarrow \infty $. We want to show that under a
specific assumption $T^x\rightarrow \infty $ (resp. $S^x \rightarrow \infty $) as $x\to\infty$, and under the complementary assumption
$\sup_{x>0} \mathbb{E} ( e^{c T^x}) <\infty $ for some $c > 0$ (resp. $\sup_{x>0} \mathbb{E} ( e^{c S^x}) <\infty$ for some $c > 0$). Because both mappings $x \mapsto  T^x$ and $x \mapsto  S^x$ are a.s. increasing, the result will follow for the same result proved for
any collection of r.v.'s $\{T^x, x > 0\}$ (resp. $\{S^x, x > 0\}$) which has the same monotonicity property, and has the same marginal laws as the original one. More precisely, we will consider the $Z^x$'s solutions of (1.2) instead of (1.1), with the same $W$ for all $x > 0$. 

We first need to recall some preliminary results on a class of one--dimensional Kolmogorov diffusions (drifted Brownian motions), which can also be found in \cite{Cat}.\\

Consider a one--dimensional drifted Brownian motion with values in $[0,\infty )$ which is killed when it first hits zero 
\begin{equation*}
d X_t = q(X_t) dt + dB_t, \qquad X_0=x>0,
\end{equation*}
where $q$ is defined and is $C^1$ on $(0,\infty )$, and $\{B_t, t\geq 0\}$ is a standard one- dimensional Brownian motion. In particular, $q$ is allowed to explode at the origin. In this section, we shall assume that \\

\textbf{Hypothesis (H2):} There exists $x_0>0$ such that $q(x)<0 \quad \forall x\geq x_0$, and
\begin{equation*}
\limsup_{x\rightarrow 0^+} q(x)<\infty.
\end{equation*}
 The condition (H2) implies that $q$ is bounded from above by some constant. It ensures that $\infty $ is inaccessible, in the sense that a.s. $\infty $ can not be reached in finite time from $X_0 = x \in (0,\infty )$.

We denote by $T_y^x$ the first time the process $X$ hits $y\in [0,\infty )$ when starting from $X_0=x$
\begin{equation*}
T_y^x= \inf\{t>0: X_t=y \mid X_0 = x\}.
\end{equation*}
We say that $\infty $ is an entrance boundary for $X$ (see, for instance, Revuz and Yor \cite{RY}, page $305$) if there is $y>0$ and a time $t>0$ such that 
\begin{equation*}
\lim_{x\uparrow \infty }\mathbb P (T_y^x<t) >0.
\end{equation*}
Let us introduce the following condition \\

\textbf{Hypothesis (H3):}
\begin{equation*}
\int_1^\infty e^{-Q(y)} \int_y^\infty  e^{Q(z)} dz dy <\infty ,
\end{equation*}
where $Q(y) = 2\int_1^y q(x) dx, y\geq 1$.

Tonelli's theorem ensures that (H3) is equivalent to
\begin{equation*}
\int_1^\infty e ^{Q(y)} \int_1^y e ^{-Q(z)} dz dy <\infty .
\end{equation*}
We have the following result which is Proposition $7.6$ in \cite{Cat}.
\begin{proposition}\label{pro3.1}
The following are equivalent:
\begin{itemize}
\item[\rm{1)}] $\infty $ is an entrance boundary for $X$.
\item[\rm{2)}] $(H3)$ holds.
\item[\rm{3)}] For any $a>0$, there exists $y_a>0$ such that
\begin{equation*}
\sup_{x>y_a} \mathbb{E} \big( e^{a T_{y_a}^x}\big) <\infty .
\end{equation*}
\end{itemize}
\end{proposition}
We now state the main result of this subsection
\begin{theorem}\label{th5}
Assume that $(H2)$ holds. We have
\begin{itemize}
\item[\rm{1)}] If $(H3)$ does not hold, then for all $y\geq 0$,
\begin{equation*}
\sup_{x>y} T_y^x = \infty \qquad a.s.
\end{equation*}
\item[\rm{2)}] If $(H3)$ holds, then for all $y\geq 0$,
\begin{equation*}
\sup_{x>y} T_y^x < \infty \qquad a.s.,
\end{equation*}
and moreover, there exists some positive constant $c$ such that
\begin{equation*}
\sup_{x>0} \mathbb{E} \big( e^{c T_{0}^x}\big) <\infty .
\end{equation*}
\end{itemize}
\end{theorem}
\begin{proof}
\begin{itemize}
\item[\rm{1)}] If (H3) does not hold, then by Proposition \ref{pro3.1}, $\infty $ is not an entrance boundary for $X$. It means that for all $y>0, t>0$,
\begin{equation*}
\lim_{x\uparrow \infty } \mathbb P (T_y^x <t) = 0.
\end{equation*}
Hence for all $t>0$, since $x\to T^x_y$ is increasing a.s.,
\begin{equation*}
\mathbb P ( \sup_{x>y} T_y^x <t) = 0,
\end{equation*}
hence
\begin{equation*}
\sup_{x>y} T_y^x =\infty  \qquad a.s.
\end{equation*}
\item[\rm{2)}] The result is a consequence of Proposition \ref{pro3.1}. We can prove it by using the same argument as used in the proof of
Theorem \ref{th2}.
\end{itemize}
\end{proof}
It is not obvious when (H3) holds. But from the following result, if $q$ satisfies some explicit conditions, we can decide whether (H3) holds or not.
\begin{proposition}\label{pro3.2}
Suppose that $(H2)$ holds. We have
\begin{itemize}
\item[\rm{1)}] If $$\int_{x_0}^\infty \frac 1{q(x)} dx = -\infty \qquad \text{and} \qquad \limsup_{x\rightarrow \infty } \frac {q^{'}(x)}{q(x)^2}<\infty ,$$ then $(H3)$ does not hold.
\item[\rm{2)}] If there exists $q_0<0$ such that $q(x)\leq q_0$ for all $x\geq x_0$, 
$$\int_{x_0}^\infty \frac 1{q(x)} dx > -\infty \qquad \text{and} \qquad \liminf_{x\rightarrow \infty } \frac {q^{'}(x)}{q(x)^2}> -2 ,$$ then $(H3)$ holds.
\item[\rm{3)}]If 
$$\int_{x_0}^\infty \frac 1{q(x)} dx > -\infty \qquad \text{and} \qquad q^{'}(x)\leq 0 \qquad \forall x\geq x_0, $$ then $(H3)$ holds.
\end{itemize}
\end{proposition}
\begin{proof}
\begin{itemize}
\item[\rm{1)}]
Define $s(y) := \int_y^\infty  e^{Q(z)} dz$. If $s(x_0)=\infty $, then $s(y)=\infty $ for all $y\geq x_0$, so that (H3) does not hold. \\
We consider the case $s(x_0)<\infty $. Integrating by parts on $\int s e^{-Q} dy$ gives
\begin{equation}\label{eq3.1}
\int_{x_0}^{\infty} s e^{-Q} dy = \int_{x_0}^{\infty} \frac {s}{2q} e^{-Q} 2q dy = \left . \frac{-s}{2q} e^{-Q} \right |_{x_0}^{\infty } - \int_{x_0}^{\infty} \frac {1}{2q}dy - \int_{x_0}^{\infty} s e^{-Q}\frac {q^{'}}{2q^2} dy
\end{equation}
 From $\int_{x_0}^\infty \frac 1{q(x)} dx = -\infty$ and $\frac{-s}{2q} e^{-Q}(\infty )\geq 0$,  \eqref{eq3.1} implies that 
$$\int_{x_0}^\infty s e^{-Q}(1+ \frac {q^{'}}{2q^2}) dy =\infty .$$
Since $\limsup_{x\rightarrow \infty } \frac {q^{'}(x)}{q(x)^2}<\infty $, then $\int_{x_0}^\infty s e^{-Q} dy =\infty $. Condition (H3) does not hold.
\item[\rm{2)}] We can easily deduce from $q(x)\leq q_0$ for all $x\geq x_0$ that $s(y)$ tends to zero as $y$ tends to infinity, and $s(y) e^{-Q(y)}$ is bounded in $y\geq x_0$. Because $\int_{x_0}^\infty \frac 1{q(x)} dx > -\infty$, \eqref{eq3.1} implies that $s e^{-Q} (1+\frac {q^{'}}{2q^2})$ is integrable. Then thanks to the condition $\liminf_{x\rightarrow \infty } \frac {q^{'}(x)}{q(x)^2}> -2$, we conclude that (H3) holds.
\item[\rm{3)}]From $q(x)\leq q(x_0)<0$ for all $x\geq x_0$, we can easily deduce that $Q(y) \rightarrow -\infty $ and $s(y)\rightarrow 0$ as $y\rightarrow \infty $. Applying the Cauchy's mean value theorem to $s(y)$ and $q_1(y) := e^{Q(y)}$, we have for all $y\geq x_0$, there exists $\xi \in (y,\infty )$ such that
\begin{equation*}
\frac{\int_y^\infty e^{Q(z)}dz}{e^{Q(y)}} = \frac{s^{'}(\xi )}{q_1^{'}(\xi )}= -\frac {1}{2q(\xi )}.
\end{equation*}
Because $q^{'}(x)\leq 0$ for all $x\geq x_0$, we obtain
\begin{equation*}
s(y) e^{-Q(y)}\leq -\frac {1}{2q(y)}, \qquad \text{for all} \qquad y\geq x_0.
\end{equation*}
Hence 
\begin{equation*}
\int_{x_0}^\infty s(y) e^{-Q(y)} dy \leq -\int_{x_0}^\infty\frac {1}{2q(y)} dy <\infty .
\end{equation*}
Then (H3) holds.
\end{itemize}
\end{proof}
\begin{example}
We are interested in the case that $q$ is a polynomial. More precisely, we consider the function $q$ satisfying (H$2$) and for all $x\geq x_0$,
$$q(x)= -x^{\alpha } \qquad \alpha > -1.$$
We have 
$$\lim_{x\rightarrow \infty } \frac {q^{'}(x)}{q(x)^2}=\lim_{x\rightarrow \infty } \frac{\alpha x^{\alpha -1}}{x^{2\alpha }} =\lim_{x\rightarrow \infty } \frac{\alpha }{x^{\alpha +1}}=0.$$
Hence condition (H3) holds if and only if 
$$\int_{x_0}^\infty \frac 1{-x^{\alpha }} dx > -\infty \Leftrightarrow \alpha >1.$$
\end{example}
\subsection{Height of the continuous forest of trees}
We consider the process $\{Z_t^x, t\geq 0\}$ solution of \eqref{eq1}. It follows from the Ito formula that the process $Y_t^x=\sqrt{Z_t^x}$ solves the SDE
\begin{align}\label{eq3.2}
dY_t^x = \frac {f((Y_t^x)^2)-1}{2Y_t^x} dt+ dW_t, \qquad Y_0^x = \sqrt{x}.
\end{align}
Note that the height of the process $Z^x$ is
\begin{equation*}
T^x= \inf\{t>0, Z_t^x= 0\} = \inf\{t>0, Y_t^x= 0\} .
\end{equation*}
We now establish the large $x$ behaviour of $T^x$.
\begin{theorem}\label{th6}
Assume that $f$ is a function satisfying (H1) and that there exists $a_0>0$ such that $f(x)\not= 0$ for all $x\geq a_0$. If $\int_{a_0}^\infty \frac 1{\mid f(x)\mid } dx =\infty$, then
 $$T^x \rightarrow \infty \quad a.s. \quad as \quad x\rightarrow \infty .$$
\end{theorem}
\begin{proof}
Let $\beta $ be a constant such that $\beta >\theta $. By a well-known comparison theorem, $Y_t^x\geq Y_t^{1,x}$, where $Y_t^{1,x}$ solves
\begin{align*}
dY_t^{1,x} = -\frac {\beta (Y_t^{1,x})^2-f((Y_t^{1,x})^2)+1}{2Y_t^{1,x}} dt+ dW_t, \qquad Y_0^{1,x} = \sqrt{x},
\end{align*}
Note that the function $\beta x -f(x)+1$ is positive and increasing, then $f_1(x):= -\frac{\beta x^2-f(x^2)+1}{2x}$ satisfies (H2), and 
$$\limsup_{x\rightarrow \infty } \frac {f_1^{'}(x)}{f_1(x)^2}<\infty.$$
Moreover there exists $x_1>0$ such that $\beta x-f(x)\ge1$ for all $x\ge x_1$, hence
\begin{align*}
\int_1^\infty \frac 1{f_1(x)} dx &= -\int_1^\infty \frac {2x}{\beta x^2-f(x^2)+1} dx \\
&=-\int_1^\infty \frac 1{\beta x-f(x)+1} dx\\
&\le-\int_1^{x_1} \frac 1{\beta x-f(x)+1} dx-2\int_{x_1}^\infty \frac 1{\beta x-f(x)}dx \\
&=-\infty,
\end{align*}
again by Lemma \ref{le2.2}. The result now follows readily from Theorem \ref{th5} and Proposition \ref{pro3.2}.
\end{proof}
\begin{theorem}\label{th7}
Assume that $f$ is a function satisfying (H1) and that there exists $a_0>0$ such that $f(x)\not= 0$ for all $x\geq a_0$. If $\int_{a_0}^\infty \frac 1{\mid f(x)\mid } dx <\infty$, then
 \begin{equation*}
\sup_{x>0} T^x < \infty \qquad a.s.,
\end{equation*}
and moreover, there exists some positive constant $c$ such that
\begin{equation*}
\sup_{x>0} \mathbb{E} \big( e^{c T^x}\big) <\infty .
\end{equation*}
\end{theorem}
\begin{proof}
We can rewrite the SDE \eqref{eq3.2} as (with again $\beta>\theta$)
\begin{align*}
dY_t^x = \frac {\beta (Y_t^x)^2 -h((Y_t^x)^2)}{2Y_t^x} dt+ dW_t, \qquad Y_0^x = \sqrt{x},
\end{align*}
where $h(x):= \beta x-f(x)+1$ is a positive and increasing function. By Lemma \ref{le2.2}, we have $\int_1^\infty \frac 1{h(x)} dx <\infty$ which is equivalent to $\sum_{n=1}^{\infty } \frac 1{h(n)}<\infty $. Let 
\begin{equation*}
a_1 = h(1), \qquad a_n = \min \{h(n), 2 a_{n-1}\} \quad \forall n>1.
\end{equation*}
It is easy to see that for all $n>1$,
\begin{equation*}
a_{n-1}<a_n \leq h(n) , \qquad \qquad \frac{a_n}{a_{n-1}}\leq 2.
\end{equation*}
We also have 
\begin{align*}
\frac 1{a_1} &= \frac 1{h(1)} \\
\frac 1{a_2} &\leq \frac 1{h(2)} +\frac 1{2a_1}=\frac 1{h(2)} +\frac 1{2h(1)} \\
\frac 1{a_3} &\leq \frac 1{h(3)} +\frac 1{2a_2} \leq \frac 1{h(3)}+\frac 1{2h(2)} +\frac 1{4h(1)}\\
&..............\\
\frac 1{a_n} &\leq \frac 1{h(n)} +\frac 1{2a_{n-1}} \leq \frac 1{h(n)}+\frac 1{2h(n-1)} +...+\frac 1{2^{n-1}h(1)}.
\end{align*}
Therefore
\begin{equation*}
\sum_{n=1}^{\infty } \frac 1{a_n}\leq 2\sum_{n=1}^{\infty } \frac 1{h(n)}<\infty .
\end{equation*}
Now, we define a continuous increasing function $g$ as follows. We first draw a broken line which joins the points $(n, a_n)$ and is the graph of $h_1$. Define the function $h_2$ as follows.
\begin{align*}
h_2(x) \quad = \quad
\begin{cases}
& h(x), \quad 0\leq x\leq 1\\
& h_1(x), \quad x\geq 1.
\end{cases}
\end{align*}
We then smoothen all the nodal points of the graph of $h_2$ to obtain a smooth curve which is the graph of an increasing function $g_1$. Let $g(x)= \frac 12 g_1(x)$. We have for all $n\geq 1$ and $x\in [n,n+1)$, $$h(x)\geq h(n)\geq a_n\geq \frac12 a_{n+1}=g(n+1)\geq g(x).$$ 
By the comparison theorem, $Y_t^x\leq Y_t^{2,x}$, where $Y_y^{2,x}$ solves
\begin{align*}
dY_t^{2,x} = \frac {\beta (Y_t^{2,x})^2 -g((Y_t^{2,x})^2)}{2Y_t^{2,x}} dt+ dW_t, \qquad Y_0^{2,x} = \sqrt{x}.
\end{align*}
Since
\begin{equation*}
\sum_{n=1}^{\infty } \frac 1{g(n)}= 2\sum_{n=1}^{\infty } \frac 1{a_n}<\infty,
\end{equation*}
we deduce that $\int_{1}^\infty \frac 1{g(x)} dx <\infty$, and $\frac{g(x)}{x} \rightarrow \infty $ as $x\rightarrow \infty $, by Lemma \ref{le2.2}. Let $f_2(x):= \frac {\beta x^2 -g(x^2)}{2x}$, then there exists $x_1>0, q_1<0$ such that $f_2(x)<q_1$ for all $x\geq x_1$, and 
\begin{equation*}
\int_{x_1}^\infty \frac 1{f_2(x)} dx= \int_{x_1}^\infty \frac {2x}{\beta x^2 -g(x^2)} dx=\int_{x_1^2}^\infty \frac {1}{\beta x -g(x)} dx>-\infty .
\end{equation*}
Moreover, 
\begin{equation*}
\liminf_{x\rightarrow \infty } \frac {f_2^{'}(x)}{f_2(x)^2}= \liminf_{x\rightarrow \infty } \frac {-4x g^{'}(x)}{g(x)^2}.
\end{equation*}
But for all $x\in [n, n+1)$,
 \begin{equation*}
\frac{g^{'}(x)x}{g(x)^2}\leq  \frac{(n+1)}{g(n)^2} \max_{i\in \{n-1, n, n+1\}} \{g(i+1)-g(i)\}<\frac{(n+1)g(n+2)}{g(n)^2}\leq \frac{4(n+1)}{g(n)}\rightarrow 0, 
\end{equation*}
as $n\rightarrow \infty $. The result follows from Theorem \ref{th5} and Proposition \ref{pro3.2}.

\end{proof}

\subsection{Total mass of the continuous forest of trees}
Recall that in the continuous case, the total mass of the genealogical tree is given as
$$S^x= \int_0^{T^x} Z_t^x dt$$
Consider the increasing process
\begin{equation*}
A^x_t= \int_0^{t} Z_s^x ds, t\geq 0,
\end{equation*}
and the associated time change 
\begin{equation*}
\eta^x (t) = \inf\{s>0, A_s>t\}.
\end{equation*}
We now define $U_t^x= \frac 12 Z^x\circ \eta^x (t) , t\geq 0$. It is easily seen that the process $U^x$ solves the SDE
\begin{equation}\label{eq3.3}
dU_t^x = \frac{f(2 U_t^x)}{4 U_t^x} dt+ dW_t, \qquad U_0^x= \frac {x}2.
\end{equation}
Let $\tau ^x := \inf\{t>0, U_t^x= 0\}$. It follows from above that $\eta^x (\tau ^x)=T^x$, hence $S^x= \tau ^x$. 
We have
\begin{theorem}\label{th8}
Suppose that the function $\frac{f(x)}{x}$ satisfies (H1) and there exists $a_0>0$ such that $f(x)\not= 0$ for all $x\geq a_0$.
\begin{itemize}
\item[\rm{1)}] If $\int_{a_0}^\infty \frac {x}{\mid f(x)\mid } dx =\infty$ then
 $$S^x \rightarrow \infty \quad a.s. \quad as \quad x\rightarrow \infty .$$
\item[\rm{2)}] If $\int_{a_0}^\infty \frac {x}{\mid f(x)\mid } dx <\infty$ then
 \begin{equation*}
\sup_{x>0} S^x < \infty \qquad a.s.,
\end{equation*}
and moreover, there exists some positive constant $c$ such that
\begin{equation*}
\sup_{x>0}\mathbb{E}\big( e^{c S^x}\big) <\infty .
\end{equation*}
\end{itemize}
\end{theorem}
\begin{proof}
Note that we can rewrite the SDE \eqref{eq3.3} as
\begin{equation*}
dU_t^x = \big( \beta U_t^x-h(U_t^x) \big) dt+ dW_t, \qquad U_0^x= \frac {x}2,
\end{equation*}
where $h(x):= \beta x-\frac {f(2x)}{4x}$, with again $\beta>\theta$, is a positive and increasing function.
\begin{itemize}
\item[\rm{1)}] By the comparison theorem, $U_t^x\geq U_t^{1,x}$, where $U_t^{1,x}$ solves
\begin{align*}
dU_t^{1,x} = -h(U_t^{1,x}) dt+ dW_t, \qquad U_0^{1,x} = \frac {x}2.
\end{align*}
The result follows from Theorem \ref{th5}, Proposition \ref{pro3.2} and Lemma \ref{le2.2}.
\item[\rm{2)}] The result is a consequence of Theorem \ref{th5} and Proposition \ref{pro3.2}. We can prove it by using the same argument as used in the proof of Theorem \ref{th7}.
\end{itemize}
\end{proof}
\section{Some examples}
In this section we will discuss some special cases to illustrate our results.
\begin{example}
An important example is the case of a logistic interaction, where 
$$f(x):= ax-bx^2, \qquad a\in \mathbb R, b\in \mathbb R_{+}.$$
There exists a positive constant $a_0$ such that $f(x)<0$ for all $x\geq a_0$, and
$$\int_{a_0}^\infty \frac {1}{\mid f(x)\mid } dx =\int_{a_0}^\infty \frac {1}{bx^2-ax}dx<\infty, \qquad \int_{a_0}^\infty \frac {x}{\mid f(x)\mid } dx=\int_{a_0}^\infty \frac {x}{bx^2-ax} dx =\infty .$$
Hence in this case, there exists some positive constant $c$ such that
\begin{align*}
\sup_{m>0} \mathbb E \big(e^{c H^m }\big) < \infty , \qquad \sup_{x>0} \mathbb E \big(e^{c T^x }\big) < \infty,
\end{align*}
and 
$$\sup_{m>0} L^m = \infty \quad a.s., \qquad \sup_{x>0} S^x = \infty \quad a.s.$$
\end{example}
\begin{example}
We consider the case where $f$ is a function satisfying (H1) and for all $x\geq 2$,
$$f(x)= -x^\alpha (\log x)^\gamma  ,\qquad \alpha \geq 0, \gamma \geq 0.$$
Note that 
\begin{align*}
\int_2^\infty \frac 1{x^\alpha (\log x)^\gamma} dx \qquad
\begin{cases}
& =\infty , \quad \text{if}\quad \alpha <1 \quad \text{or} \quad \alpha =1, \gamma \leq 1\\
& <\infty , \quad \text{if}\quad \alpha >1 \quad \text{or} \quad \alpha =1, \gamma > 1.
\end{cases}
\end{align*}
Hence 
$$\sup_{m>0} H^m = \infty \quad a.s., \qquad \sup_{x>0} T^x = \infty \quad a.s.$$
if $\alpha <1$ or $\alpha =1, \gamma \leq 1$, while there exists some positive constant $c$ such that
\begin{align*}
\sup_{m>0} \mathbb E \big(e^{c H^m }\big) < \infty , \qquad \sup_{x>0} \mathbb E \big(e^{c T^x }\big) < \infty
\end{align*}
if $\alpha >1$ or $\alpha =1, \gamma > 1$. Concerning the length (resp. the total mass) of the genealogical forest of trees we have
$$\sup_{m>0} L^m = \infty \quad a.s., \qquad \sup_{x>0} S^x = \infty \quad a.s.$$
if $\alpha <2$ or $\alpha =2, \gamma \leq 1$, while there exists some positive constant $c$ such that
\begin{align*}
\sup_{m>0} \mathbb E \big(e^{c L^m }\big) < \infty , \qquad \sup_{x>0} \mathbb E \big(e^{c S^x }\big) < \infty
\end{align*}
if $\alpha >2$ or $\alpha =2, \gamma > 1$.
\end{example}
\paragraph{Acknowledgement} The authors want to thank an anonymous Referee, whose remarks
resulted in a significant improvement of the paper.

\bibliographystyle{plain}

\end{document}